\documentclass[]{amsart}

\usepackage[T1]{fontenc}	
\usepackage{totpages}
\usepackage{graphicx}
\usepackage{amsmath, amsthm, amssymb}
\usepackage{upgreek}
\usepackage{mathrsfs}
\usepackage{pdfsync}
\usepackage[usenames,dvipsnames]{xcolor}
\usepackage[inline,shortlabels]{enumitem}
\usepackage[hyphens]{url}									%per gli url in bliografia
\usepackage{verbatim}								%per commentare cose lunghe
\usepackage{dsfont}									%doublestroke \mathds
\usepackage{bbm}

\usepackage{mathtools}
\usepackage{xparse}
\usepackage{microtype}

\usepackage[font=small]{caption}

\usepackage[pdftex,bookmarks,
						hidelinks,					%nasconde i link: antagonista di colorlinks (s 1 E no 2)
						breaklinks,%unicode,
						citecolor=OliveGreen,
						linkcolor=Maroon
						]{hyperref} %hyperref deve stare per ultimo

\usepackage{mathtools}

\usepackage[smalltableaux,aligntableaux=center]{ytableau}

\allowdisplaybreaks

\usepackage{xparse}

\ExplSyntaxOn
\NewDocumentCommand{\makeabbrev}{mmm}
 {
  \yoruk_makeabbrev:nnn { #1 } { #2 } { #3 }
 }

\cs_new_protected:Npn \yoruk_makeabbrev:nnn #1 #2 #3
 {
  \clist_map_inline:nn { #3 }
   {
    \cs_new_protected:cpn { #2 } { #1 { ##1 } }
   }
 }
 \ExplSyntaxOff

\makeabbrev{\textbf}{tbf#1}{a,b,c,d,e,f,g,h,i,j,k,l,m,n,o,p,q,r,s,t,u,v,w,x,y,z,A,B,C,D,E,F,G,H,I,J,K,L,M,N,O,P,Q,R,S,T,U,V,W,X,Y,Z}

\makeabbrev{\textbf}{bf#1}{a,b,c,d,e,f,g,h,i,j,k,l,m,n,o,p,q,r,s,t,u,v,w,x,y,z,A,B,C,D,E,F,G,H,I,J,K,L,M,N,O,P,Q,R,S,T,U,V,W,X,Y,Z}

\makeabbrev{\textsf}{tsf#1}{a,b,c,d,e,f,g,h,i,j,k,l,m,n,o,p,q,r,s,t,u,v,w,x,y,z,A,B,C,D,E,F,G,H,I,J,K,L,M,N,O,P,Q,R,S,T,U,V,W,X,Y,Z}

\makeabbrev{\mathsf}{mss#1}{a,b,c,d,e,f,g,h,i,j,k,l,m,n,o,p,q,r,s,t,u,v,w,x,y,z,A,B,C,D,E,F,G,H,I,J,K,L,M,N,O,P,Q,R,S,T,U,V,W,X,Y,Z}

\makeabbrev{\mathfrak}{mf#1}{a,b,c,d,e,f,g,h,i,j,k,l,m,n,o,p,q,r,s,t,u,v,w,x,y,z,A,B,C,D,E,F,G,H,I,J,K,L,M,N,O,P,Q,R,S,T,U,V,W,X,Y,Z,
sl,gl
}

\makeabbrev{\mathrm}{mrm#1}{a,b,c,d,e,f,g,h,i,j,k,l,m,n,o,p,q,r,s,t,u,v,w,x,y,z,A,B,C,D,E,F,G,H,I,J,K,L,M,N,O,P,Q,R,S,T,U,V,W,X,Y,Z}

\makeabbrev{\mathbf}{mbf#1}{a,b,c,d,e,f,g,h,i,j,k,l,m,n,o,p,q,r,s,t,u,v,w,x,y,z,A,B,C,D,E,F,G,H,I,J,K,L,M,N,O,P,Q,R,S,T,U,V,W,X,Y,Z}

\makeabbrev{\mathcal}{mc#1}{A,B,C,D,E,F,G,H,I,J,K,L,M,N,O,P,Q,R,S,T,U,V,W,X,Y,Z}

\makeabbrev{\mathbb}{mbb#1}{A,B,C,D,E,F,G,H,I,J,K,L,M,N,O,P,Q,R,S,T,U,V,W,X,Y,Z}

\makeabbrev{\mathscr}{ms#1}{A,B,C,D,E,F,G,H,I,J,K,L,M,N,O,P,Q,R,S,T,U,V,W,X,Y,Z}

\makeabbrev{\mathrm}{#1}{
Id,id,ran,rk,diag,stab,ann,conv,pr,ev,tr,End,Hom,sgn,im,op,can,fin,ext,red,tot,Leb,lex,Aut,Inn,
rot,usc,lsc,Lip,lip,bSymLip,osc,AC,loc,coz,z,
supp,Opt,Adm,Cpl,Geo,GeoOpt,GeoAdm,GeoCpl,reg,res,graph,
bd,co,Ric,Exp,dExp,dist,seg,Seg,cut,fcut,Cut,SDiff,Iso,Isom,diam,cl,Homeo,Diff,Der,vol,dvol,inj,relint, Graph, sub,
var,law,Var,Poi,Gam,pa,so,iso,fs,inv,pqi,mix,erg,
TestF,
ob,dom,cod,inp,
}

\makeabbrev{\mathsf}{#1}{CD,BE,MCP,Ent,wMTW,MTW,Ch,RCD,EVI,Rad,dRad,SL,cSL,dSL,ScL,Irr,SC,wFe,VA,MetMeas,UMeas,CSMet,Met,USp,Meas,Mbl,alg,Alg}

\makeabbrev{\mathsc}{#1}{mmaf,cg}

\newcommand{\A}{\Sigma} %SIGMA-ALGEBRA
\newcommand{\mathsc}[1]{\text{\textsc{#1}}}

\DeclareMathOperator{\eqdef}{\coloneqq}

\let\epsilon\varepsilon

\newcommand{\tparen}[1]{\big({#1}\big)}

\newcommand{\class}[2][]{\left[#2\right]_{#1}}						%Measure classes
\newcommand{\pfwd}{\sharp}

\DeclareMathOperator{\car}{\mathbf 1}

\DeclareMathOperator{\emp}{\varnothing}

\newcommand{\comma}{\,\mathrm{,}\;\,}

\newcommand{\fstop}{\,\mathrm{.}}

\usepackage{scrextend}						%KOMA: confligge quindi deve stare qui

\let\temp\phi
\let\phi\varphi
\let\varphi\temp

\numberwithin{equation}{section}
\theoremstyle{plain}
\newtheorem{theorem}{Theorem}%[section]
\newtheorem*{theorem*}{Theorem}
\newtheorem*{mthm*}{Main Theorem}

\newtheorem{proposition}[theorem]{Proposition}%[section]
\newtheorem{corollary}[theorem]{Corollary}%[section]

\theoremstyle{definition}
\newtheorem*{defs*}{Definition}%[section]
\theoremstyle{remark}
\newtheorem{remark}[theorem]{Remark}%[section]
\newtheorem*{ass*}{Assumption}%[section]

\renewcommand{\paragraph}[1]{\medskip\emph{#1}.\quad}

\begin{document}

\title[A characterization of Maps of Bounded Compression]{A characterization of\\Maps of Bounded Compression}
\thanks{The author gratefully acknowledges funding of his current position by the Austrian Science Fund (FWF), grant ESPRIT208.
He is grateful to Enrico Pasqualetto for pointing out some references on maps of bounded compression.}

\author[L.~Dello Schiavo]{Lorenzo Dello Schiavo}
\address{Institute of Science and Technology Austria
\\
Am Campus 1\\
3400 Klosterneuburg\\
Austria
}
\email{lorenzo.delloschiavo@ist.ac.at}

\begin{abstract}
A measurable map between measure spaces is shown to have bounded compression if and only if its image via the measure-algebra functor is Lipschitz-continuous w.r.t.\ the measure-algebra distances.
This provides a natural interpretation of maps of bounded compression/deformation by means of the measure-algebra functor and corroborates the assertion that maps of bounded deformation are the natural class of morphisms for the category of complete and separable metric measure spaces.
\end{abstract}

\maketitle

%\begin{center}
%\today
%\end{center}
%
%\setcounter{tocdepth}{3}
%\tableofcontents
\section{Introduction}
Let~$\phi\colon (X_1,\A_1,\mu_1)\to(X_2,\A_2,\mu_2)$ be a measurable map between two measure spaces, and denote by~$\phi_\pfwd\mu_1\eqdef \mu_1\circ \phi^{-1}$ the push-forward measure of~$\mu_1$ via~$\phi$.
We say that~$\phi$ is \emph{inverse-nil-preserving} if 
\[
\phi_\pfwd\mu_1\ll\mu_2\comma
\]
and that~$\phi$ has \emph{bounded compression} if there exists a constant~$C=C_\phi\in (0,\infty)$ such that
\[
\phi_\pfwd \mu_1\leq C\mu_2 \quad \text{on} \quad \A_2\fstop
\]
We call the infimal such constant the \emph{compression} of~$\phi$.
Maps of bounded compression ---in this generality firstly considered by N.~Gigli in~\cite{Gig18}--- have found numerous applications in metric-measure-space analysis, where they play a key role in several important definitions.
Notably, they are instrumental to the definition of \emph{minimal weak upper gradient} of a real-valued function on a metric measure space~\cite{AmbGigSav14}, and of \emph{pull-back} of \emph{normed modules}~\cite{Gig18}, also cf.~\cite[Chap.~3]{GigPas20}.

In spite of their importance, it seems however that maps of bounded compression have not been much investigated as a measure-theoretical construct in their own right, i.e.\ when no distance is involved.
Here, we fill this gap by unveiling the meaning, significance, and naturality of the notion of bounded compression in the category of measure algebras.
Our main result may be informally stated as follows: 
\begin{theorem*}
A map has bounded compression if and only if its image via the measure-algebra functor is Lipschitz-continuous with respect to the measure-algebra distances.
\end{theorem*}

The significance of maps of bounded compression in the category of measure algebras is thus a consequence of the naturality (in the non-technical sense) of the measure-algebra functor.

Let us now recall the necessary definitions, following~\cite[Vol.~III]{Fre00}.

\subsection{Some categories}\label{s:Categories}
We say that~$\mssC$ is \emph{a} category of objects~$\msO$, if the objects~$\ob(\mssC)$ of~$\mssC$ coincide with~$\msO$ and the morphisms~$\hom(\mssC)$ of~$\mssC$ are unassigned.
Let
\begin{itemize}
\item $\USp$ be the category of uniform spaces and uniformly continuous maps;
\item $\Met$ be the category of complete and separable metric spaces with all \emph{uniformly continuous maps} as morphisms;
\item $\Met_b$ be the category of complete and separable metric spaces with all \emph{Lipschitz-continuous maps} as morphisms;
\item $\Met_1$ be the category of complete and separable metric spaces with all \emph{short\footnote{We say that a map is \emph{short} if it is Lipschitz-continuous with Lipschitz constant less than~$1$.} maps} as morphisms;
\item $\Meas$ be a category of triples~$\mbbX\eqdef (X,\A,\mu)$ with~$(X,\A)$ a standard Borel space and~$\mu$ a $\sigma$-finite measure on~$(X,\A)$, and morphisms~$\msA\eqdef\hom(\Meas)$. 
We require each~$\phi\in \hom(\mbbX_1,\mbbX_2)$ to be a measurable map~$\phi\colon X_1\to X_2$.
We write~$\msA_{\inp}$ for the subclass of~$\msA$ consisting of inverse-nil-preserving maps.
\end{itemize}

\subsection{Measure algebras}
For~$\mbbX\in\ob(\Meas)$, let~$(\mfA,\bar\mu)$ be the \emph{measure algebra} of~$(\A,\mu)$, that is, the Boolean algebra of equivalence classes of sets in~$\A$ modulo $\mu$-null sets, endowed with the quotient measure functional~$\bar\mu$, e.g.~\cite[321H-I]{Fre00}.
Whenever no confusion may arise, we suppress~$\bar\mu$ from the notation, just writing~$\mfA$ for the measure algebra of~$\mbbX$.
It is always possible to endow~$\mfA$ with a uniformity of pseudo-metrics~$\msU$, turning it into a uniform space on which the standard Boolean-algebra operations are uniformly continuous, e.g.~\cite[323A(b), 323B]{Fre00}. (For uniform spaces and uniformities, see e.g.~\cite[3A4]{Fre00}.)

Consider now a morphism~$\phi\in\hom(\mbbX_1,\mbbX_2)$.
We write~$\phi\in\hom_{\inp}(\mbbX_1,\mbbX_2)$ to indicate that~$\phi$ is additionally inverse-nil-preserving.
In this case, the map~$\phi$ descends to a Boolean homomorphism~$\phi^\bullet\colon\mfA_2\to\mfA_1$, e.g.~\cite[324B]{Fre00}, defined by
\begin{align}\label{eq:Dagger}
\phi^\bullet\colon\class[2]{A}\mapsto \class[1]{\phi^{-1}(A)}\fstop
\end{align}
In the next proposition we summarize the virtually well-known construction of the \emph{measure-algebra functor}~$\Alg$ on~$\Meas$ defined by
\[
\Alg\colon \mbbX\longmapsto (\mfA,\msU) \qquad \text{and}\qquad \Alg\colon \phi\longmapsto \phi^\bullet\fstop
\]
\begin{proposition}\label{p:USp}
The following assertions are equivalent:
\begin{enumerate}[$(i)$]
\item\label{i:p:USp:1} every morphism of~$\Meas$ is inverse-nil-preserving (i.e.,~$\msA=\msA_{\inp}$);
\item\label{i:p:USp:2} $\Alg$ is a (contravariant) functor on~$\Meas$ with values in~$\USp$.
\end{enumerate}
\end{proposition}

\begin{remark}
The assertion in Proposition~\ref{p:USp} is non-quantitative and may in fact be rephrased without any reference to measures.
Indeed, we might have alternatively stated it for a category with objects~$(X,\A,\mcN)$ with~$(X,\A)$ a standard Borel space, and~$\mcN$ a $\sigma$-ideal of~$\A$ ---playing the role of the $\sigma$-ideal~$\mcN_\mu$ of $\mu$-null sets of $\sigma$-finite measure~$\mu$ on~$(X,\A)$.
This motivated our choice of terminology for \emph{inverse-nil-preserving} maps, since~$\phi\colon \mbbX_1\to\mbbX_2$ is inverse-nil-preserving precisely when~$\phi^{-1}(\mcN_2)\subset \mcN_1$.
\end{remark}

Under the additional datum of a uniform structure on objects of~$\Meas$, Proposition~\ref{p:USp} may be used to characterize uniformly continuous inverse-nil-preserving maps via~$\Alg$ and the forgetful functor to~$\USp$.
Indeed, let~$\UMeas$ be a category of complete and separable uniform spaces~$(X,\msU)$ endowed with $\sigma$-finite Borel measures~$\mu$, and denote by~$\mssf$ the map on~$\UMeas$ defined on objects by~$\mssf\colon (X,\msU,\mu)\to (X,\msU)\in\ob(\USp)$ and preserving morphisms.
Then,

\begin{corollary}
The following assertions are equivalent:
\begin{enumerate}[$(i)$]
\item every morphism of $\UMeas$ is uniformly continuous and inverse-nil-preserving;
\item $\Alg$ and~$\mssf$ are functors on~$\UMeas$ with values in~$\USp$.
\end{enumerate}
\end{corollary}

\subsection{Main result}
Relying on maps of bounded compression, we now turn to a quantitative version of Proposition~\ref{p:USp}.
Let~$\mfA$ be the measure algebra of~$\mbbX\in\ob(\Meas)$.
Write~$\mfA^\fin$ for the ideal of~$\mfA$ consisting of elements with finite $\bar\mu$-measure, and note that the quantity
\[
\rho(a,b)\eqdef \bar\mu(a\triangle b)\comma \qquad a,b\in\mfA\comma
\]
defines a distance~$\rho$ on~$\mfA^\fin$, e.g.~\cite[323A(e)]{Fre00}.

In order to state our main result, we define a map~$\alg$ on~$\Meas$ by
\[
\alg\colon \mbbX\longmapsto(\mfA^\fin,\rho) \qquad \text{and} \qquad \alg\colon \phi\longmapsto\phi^\bullet\fstop
\]

\begin{theorem}\label{t:Main}
The following assertions are equivalent:
\begin{enumerate}[$(i)$]
\item\label{i:t:Main:1} every morphism of $\Meas$ has bounded compression;
\item\label{i:t:Main:2} $\alg$ is a functor on~$\Meas$ with values in~$\Met_b$.%, i.e.~$\alg(\phi)$ and~$\mssf(\phi)$ are Lipschitz maps for all~$\phi\in\msA$.
\end{enumerate}
\end{theorem}

\subsection{A natural choice of morphisms for metric measure spaces}
After the work of J.R.~Isbell~\cite{Isb64}, \emph{the} category of metric spaces is usually defined to have all short maps as morphisms (giving rise to~$\Met_1$ in~\S\ref{s:Categories}).
This is essentially the same as choosing as morphisms the class of all Lipschitz-continuous maps (giving rise to~$\Met_b$), in that any Lipschitz-continuous map may be turned into a short map by linearly rescaling distances, and such rescaling has nice categorical properties.
Occasionally, the larger class of uniformly continuous maps too is chosen as the class of morphisms of a category of metric spaces (giving rise to~$\Met$), since uniform continuity is a minimal requirement in discussing the preservation of, e.g., completeness.
This ambiguity for the choice of morphisms in a category of metric measure spaces may be resolved by introducing some additional structure.
Below, we show that when each~$(X,\mssd)\in \ob(\Met_b)$ (which is the same as~$\ob(\Met)$ and~$\ob(\Met_1)$) is further endowed with a $\sigma$-finite Borel measure, then there is a natural choice of morphisms for the result category, namely all Lipschitz-continuous maps of bounded compression.

Indeed, let~$\MetMeas$ be a category of triples~$\mbbX\eqdef (X,\mssd,\mu)$ with~$(X,\mssd)$ a complete and separable metric space and~$\mu$ a $\sigma$-finite Borel measure on~$(X,\mssd)$, and morphisms~$\msB\eqdef\hom(\MetMeas)$ with~$\phi\in \hom(\mbbX_1,\mbbX_2)$ Borel measurable and inverse-nil-preserving.
Denote by~$\mcB$ the Borel $\sigma$-algebra of~$\mbbX\in \ob(\MetMeas)$, and note that~$(X,\mcB)$ is a standard Borel space since~$(X,\mssd)$ is complete and separable.
Thus,~$\mssg\colon \mbbX\mapsto (X,\mcB,\mu)$ maps~$\ob(\MetMeas)\to \ob(\Meas)$ and allows us to identify morphisms in~$\msB$ as morphisms between objects of~$\Meas$.
Under this identification, we may therefore compare the morphisms~$\msB$ of~$\MetMeas$ with those~$\msA$ of~$\Meas$.
If~$\msB\subset\msA$, then~$\mssg$ is a (forgetful) functor~$\MetMeas\to\Meas$ and it is further essentially surjective, since every~$(X,\A,\mu)\in\ob(\Meas)$ arises as the standard Borel $\sigma$-finite measure space associated to an object~$(X,\mssd,\mu)$ by definition of standard Borel space and forgetting the assignment of the distance~$\mssd$ on~$X$.
Thus, if~$\msB=\msA$, the functor~$\mssg$ is an equivalence of categories.
In the following, we shall therefore ---with no loss of generality--- deal with a category~$\MetMeas$ with same morphisms~$\msB=\msA$ as $\Meas$.

Denote now by~$\mssf$ the forgetful functor from~$\MetMeas$ to a category of metric spaces, mapping~$\mbbX$ to~$(X,\mssd)$ and preserving morphisms.
Clearly,~$\mssf\colon \ob(\MetMeas)\to\ob(\Met_b)$, and~$\mssf\colon \msA\to\hom(\Met_b)$ if and only if~$\msA$ consists of Lipschitz-continuous maps.
After~\cite[Dfn.~2.4.1]{Gig18}, we say that a map~$\phi\colon \mbbX_1\to \mbbX_2$ has \emph{bounded deformation} if it is both Lipschitz and of bounded compression.
Again in light of the equivalence of~$\MetMeas$ and~$\Meas$ we may as well regard~$\alg$ as a map on~$\MetMeas$.
Thus, we also have:
\begin{corollary}\label{c:Main}
The following assertions are equivalent:
\begin{enumerate}[$(i)$]
\item every morphism of $\MetMeas$ has bounded deformation;
\item $\alg$ is a functor on~$\MetMeas$ and both $\mssf$ and~$\alg$ take values in~$\Met_b$.
\end{enumerate}
\end{corollary}

We note that the requirement of~$\phi\in\msA$ having bounded compression \emph{competes} with that of~$\phi$ being Lipschitz.
For instance, a constant map~$x$ is `as much Lipschitz as possible' (since its Lipschitz constant is zero), but its compression is `maximally unbounded' (since~$x_\pfwd \mu=(\mu X)\delta_x$ is a multiple of a Dirac mass).
More precisely ---as we will show in the proof of Theorem~\ref{t:Main}--- a map~$\phi\colon (X_1,\mu_1)\to (X_2,\mu_2)$ has  compression~$C$ if and only if~$\phi^\bullet\colon (\mfA_2,\bar\mu_2)\to (\mfA_1,\bar\mu_1)$ is $C$-Lipschitz.
The above competition is thus a consequence of the fact that~$\mssf$ is covariant, while~$\alg$ is contravariant.

Informally, Corollary~\ref{c:Main} resolves the competition between being Lipschitz-continuous and having bounded compression by showing that, when a category~$\MetMeas$ can be understood as a subcategory of~$\Met_b$ via the $\alg$ functor, then all maps have bounded compression, and thus that maps of bounded deformation are a natural class of morphisms for a category with objects~$\ob(\MetMeas)$.

\section{Proofs}
\begin{proof}[Proof of Proposition~\ref{p:USp}]
\ref{i:p:USp:2}$\implies$\ref{i:p:USp:1}
It suffices to note that, by~\cite[324B]{Fre00}, the Boolean homomorphism~$\Alg(\phi)\eqdef \phi^\bullet$ is well-defined (if and) only if~$\phi$ is inverse-nil-preserving.

\ref{i:p:USp:1}$\implies$\ref{i:p:USp:2}
As discussed above,~$\Alg\colon \ob(\Meas)\to\ob(\USp)$.
Thus, since~$\msA=\msA_{\inp}$ by assumption, then~$\Alg$ is a functor on~$\Meas$ by~\cite[324C(c), 324D]{Fre00}.
It remains to show that~$\Alg\colon \msA_{\inp}\to\hom(\USp)$, i.e.\ that~$\phi^\bullet\colon \msA_2 \to \msA_1$ is $\msU_2/\msU_1$-uniformly continuous for every~$\phi\in\hom_{\inp}(\mbbX_1,\mbbX_2)$.
To this end, we argue as follows.
Since every~$\mbbX\in\ob(\Meas)$ is a $\sigma$-finite standard Borel spaces, its measure algebra satisfies the countable chain condition~\cite[316A]{Fre00} by combining~\cite[322B(c) and 322G]{Fre00}.
In light of the countable chain condition, the sequential order-continuity of a Boolean homomorphism on~$\mfA$ coincides with its order-continuity by~\cite[316F(d)]{Fre00}, and in turn with its uniform continuity by~\cite[324F(a)]{Fre00}.
Therefore, it suffices to show that~$\phi^\bullet$ is sequentially order-continuous, which is shown in~\cite[324B]{Fre00}.
\end{proof}

\begin{proof}[Proof of Theorem~\ref{t:Main}]
We show that~$\alg\colon \ob(\Meas)\to\ob(\Met_b)$ and~$\alg\colon \msA\to\hom(\Met_b)$ if and only if every~$\phi\in\msA$ has bounded compression.
 
\ref{i:t:Main:1}$\implies$\ref{i:t:Main:2}.
For every~$\mbbX\in\ob(\Meas)$ the algebra $(\mfA^\fin,\rho)$ is a complete metric space by~\cite[323X(g)]{Fre00}.
Note that, since~$\mu$ is a $\sigma$-finite measure on a standard Borel space,~$L^1(\mu)$ is separable, e.g.~\cite[365X(p)]{Fre00}.
Again by~\cite[323X(g)]{Fre00}, the map~$\chi\colon\mfA^\fin\to L^0(\mu)$ defined by~$\chi(\class{A})=\class[\mu]{\car_A}$ is an isometry of~$(\mfA^\fin,\rho)$ into~$L^1(\mu)$.
Thus,~$(\mfA^\fin,\rho)$ is separable, being (isometric to) a subset of the separable \emph{metric} space~$L^1(\mu)$.
As a consequence,~$\alg(\mbbX)=(\mfA^\fin,\rho)\in\ob(\Met_b)$ for every~$\mbbX\in\ob(\Meas)$.

Now, let~$\phi\colon \mbbX_1\to\mbbX_2$ have compression~$C$.
Then, for all~$A\in\A_2$,
\[
\bar\mu_1\phi^\bullet \class[2]{A}= (\mu\circ\phi^{-1})A=\phi_\pfwd\mu_1 A \leq C \mu_2 A= C \bar\mu_2\class[2]{A}\comma
\]
which shows that~$\phi^\bullet (\mfA_2^\fin)\subset \mfA_1^\fin$, i.e.\ that~$\alg(\phi)\colon \mfA_2^\fin\to \mfA_1^\fin$ is a map between the right objects.
Furthermore, for all~$A,B\in\A_2$,
\begin{align*}
\rho_1\tparen{\phi^\bullet\class[2]{A},\phi^\bullet\class[2]{B}}=&\ \bar\mu_1\tparen{\phi^\bullet\class[2]{A}\triangle \phi^\bullet\class[2]{B}}= \mu_1 \tparen{\phi^{-1}(A) \triangle \phi^{-1}(B)} = \phi_\pfwd\mu_1 (A\triangle B)
\\
\leq&\ C \mu_2(A\triangle B)= C\bar\mu_2(\class[2]{A}\triangle \class[2]{B})= C\rho_2(\class[2]{A},\class[2]{B}) \fstop
\end{align*}
Thus ---all other necessary verifications being straightforward---,~$\alg$ is indeed a $\Met_b$-valued functor.

\ref{i:t:Main:2}$\implies$\ref{i:t:Main:1}.
Let~$\phi\in\msA$.
Since~$\alg$ is a functor on~$\MetMeas$ with values in~$\Met_b$, then~$\phi^\bullet\in\hom(\mfA_2^\fin,\mfA_1^\fin)$ is a $\rho_2/\rho_1$ Lipschitz map.
Let~$C$ be the Lipschitz constant of~$\phi^\bullet$. 
Then, for all~$A,B\in\A_2$,
\begin{align*}
\phi_\pfwd\mu_1 (A\triangle B)=&\ \mu_1 \tparen{\phi^{-1}(A) \triangle \phi^{-1}(B)} =\bar\mu_1\tparen{\phi^\bullet\class[2]{A}\triangle \phi^\bullet\class[2]{B}}
= \rho_1\tparen{\phi^\bullet\class[2]{A},\phi^\bullet\class[2]{B}}
\\
\leq&\ C\rho_2(\class[2]{A},\class[2]{B})=C\bar\mu_2(\class[2]{A}\triangle \class[2]{B})=C \mu_2(A\triangle B)\comma
\end{align*}
and choosing~$B=\emp$ (i.e.~$A\triangle B=A$) shows that~$\phi$ has compression~$C$.
\end{proof}

%{\small
%\bibliographystyle{abbrvurl}
%\bibliography{/Users/lorenzodelloschiavo/Documents/MasterBib.bib}
%}

\end{document}